\documentclass[a4paper,12pt]{article}

\usepackage{amsthm}
\usepackage{amsmath}
\usepackage{amssymb}
\usepackage{graphicx}
\usepackage[round]{natbib}
\usepackage{bbm}
\usepackage[margin=2.5cm]{geometry}

\usepackage{color}
\usepackage{framed}
\usepackage{comment}
\definecolor{shadecolor}{gray}{0.9}

\theoremstyle{plain}  
\newtheorem{theorem}{Theorem}[section] 
\newtheorem{lemma}[theorem]{Lemma}

\theoremstyle{definition} 
\newtheorem{definition}[theorem]{Definition}

\newtheorem{example}[theorem]{Example}
\newtheorem{assumption}{Assumption}

\theoremstyle{remark} 
\newtheorem{remark}[theorem]{Remark}

\newcommand{\prob}{\mathsf{P}}

\newcommand{\diff}{\,\mathrm{d}}
\newcommand{\E}{\mathsf{E}}

\newcommand{\R}{\mathbb{R}}

\renewcommand{\subset}{\subseteq}

\begin{document}

\title{T-calibration in semi-parametric models}
\author{Anja M\"uhlemann\thanks{University of Bern, \texttt{anja.muehlemann@unibe.ch}}\;  and Johanna Ziegel\thanks{ETH Zurich, \texttt{ziegel@stat.math.ethz.ch}}}
\maketitle

\begin{abstract}
This note relates the calibration of models to the consistent loss functions for the target functional of the model. We demonstrate that a model is calibrated if and only if there is a parameter value that is optimal under all consistent loss functions.
\end{abstract}


\section{Introduction}

When probability predictions $p \in [0,1]$ for the success of a binary outcome $Y \in \{0,1\}$ are issued, it is broadly accepted that such predictions should be \emph{calibrated} in the sense that
\[
\prob(Y = 1 \mid p) = p, \quad \text{almost surely.}
\]
Similarly, when probabilistic predictions in the form of a distribution over the possible values of the outcome $Y$ are issued, these predictions ought to be calibrated in a suitable sense in order to ensure the statistical compatibility between predictions and observations \citep{Dawid1984,DieboldGuntherETAL1998,StrahlZiegel2017,GneitingResin2023}. More recently, it has also been argued that point-predictions for statistical functionals, and respective statistical models should be calibrated in a suitable sense \citep{GneitingResin2023,WuthrichZiegel2024}. 

In this introduction, we will assume that the mean or expectation is the functional of interest but the results of the paper are for a general identifiable (and elicitable) functional such as quantiles, expectiles, moments, robust location functionals, or ratios of expectations. Let $(X,Y)$ be a random vector with $X \in \mathcal{X}$, $Y \in \R$ and assume that $Y$ has finite mean.  A model, $m(x;\beta)$, $\beta \in \Theta$, is \emph{(expectation-)calibrated} if
\begin{equation*}
\E\left( Y|m(X;\beta^\ast)\right) =m(X;\beta ^{\ast }) \quad 
\text{a.s.~for some }\beta ^{\ast }\in  \Theta.
\end{equation*}
Clearly, if the model is correctly specified, that is,
\begin{equation}\label{eq:cor_spec}
\E\left( Y|X\right) =m(X;\beta ^{\ast }) 
\quad \text{a.s.~for some }\beta ^{\ast }\in  \Theta,
\end{equation}
it is also calibrated. However, the converse may not hold. This can easily be seen by taking a constant model $m(X;\beta) = m(\beta)$ that does not depend on the covariate $X$. Then, the model is calibrated as soon as $m(\beta)$ takes the value $\E[Y]$ for some $\beta \in \Theta$. Expectation-calibration of a model is related to the concept of self-consistency, whose connection to regression is discussed in \citep[Section 3]{TarpeyFlury1996}. 

In this paper, we consider the situation where it is desired to remain agnostic about the conditional distribution of $Y$ given $X$. This situation is often termed semi-parametric since the model for the conditional mean is parametric, whereas there are no further assumptions on the conditional distribution of $Y$ given $X$ (other than the necessary assumptions ensuring that the conditional functionals even exist). Whether or not a model is calibrated (and the correct parameter $\beta^\ast$ has been found at least approximately by $\hat{\beta}$), can be checked empirically based only on model predictions $m(x_1;\hat{\beta}),\dots,m(x_n;\hat{\beta})$ and corresponding outcomes $y_1,\dots,y_n$; see \citet{GneitingResin2023}. This is in contrast to assessing the correct specification of a model that is typically much harder.

A Bregman loss function is a function of the form
\begin{equation}\label{eq:Bregman}
L_\phi(z,y) = \phi(y) - \phi(z) - \phi'(z)(y-z),\quad z,y \in \R,
\end{equation}
where $\phi$ is a convex function with subgradient $\phi'$. If $\phi$ is strictly convex, $\phi(Y)$ has finite mean, and the model is correctly specified as at \eqref{eq:cor_spec} with a unique correct parameter $\beta^\ast$, we obtain that
\begin{equation}\label{eq:beta_ast}
    \beta^\ast =  \arg\min_{\beta \in \Theta} \E L_\phi(m(X;\beta),Y).
\end{equation}
The Bregman loss functions at \eqref{eq:Bregman} are all consistent loss functions for the mean in the sense of \citet[Definition 1]{Gneiting2011}. More precisely, \citet{Savage1971} has shown that, under mild regularity conditions, any loss function $L$ that satisfies $\E (L(\E Y,Y)) \le \E(L(x,Y))$ for all $x \in \R$ and all distributions of $Y$ with compact support, has to be of the form \eqref{eq:Bregman} for a convex function $\phi$. Furthermore, $\E (L(\E Y,Y)) = \E(L(x,Y))$ implies $x = \E Y$ if and only if, additionally, $\phi$ is strictly convex; see also \citet[Theorem 7]{Gneiting2011}. 

Equation \eqref{eq:beta_ast} suggests M-estimators for the unknown parameter $\beta^\ast$, and indeed, 
under correct specification, minimizing the sample mean of a Bregman loss 
\begin{equation}\label{eq:M-est}
\hat{\beta}_{\phi ,n}:= \arg \min_{\beta }\frac{1}{n}\sum_{i=1}^{n}L_\phi\left( y_i,m\left( x_i;\beta \right)\right) 
\end{equation}
yields a consistent estimator of $\beta ^{\ast }$ for \emph{any} choice of
Bregman loss function, subject to moment conditions. Conversely, \citet{DimitriadisFisslerETAL2024} have shown that the only loss functions that yield consistent M-estimators as in \eqref{eq:M-est}, without substantially restricting the class of distributions for $Y$, are Bregman losses.

For any strictly convex $\phi$, let us define $\beta(\phi) =\arg\min_{\beta \in \Theta} \E L_\phi(m(X;\beta),Y)$ and assume that this $\arg\min$ is unique for simplicity. Then, whether or not the model is correctly specified, we have that $\hat{\beta}_{\phi,n}$ is consistent for $\beta(\phi)$ under regularity conditions. Under correct model specification, $\beta(\phi) = \beta^\ast$ for all $\phi$. Under model misspecification, it is generally (though not always) the case that $\beta(\phi_1)\neq \beta(\phi_{2})$. This also affects estimation, since the model is typically not correctly specified for the empirical distribution of the data, and hende, the parameter estimate is sensitive to the choice of loss function used in estimation \citep{Patton2020}. 

In this paper, we relate the calibration of models to the optimal parameters $\beta(\phi)$ under different loss functions. It turns out that if there is a parameter $\beta^*$ that is preferred under all consistent loss functions, that is, under any Bregman loss function in the case of the mean, then the model is calibrated. We will provide the characterization not only for models of the conditional mean but for general models for conditional identifiable and elicitable functionals including quantiles and expectiles; see Section \ref{sec:T-cal}. Our approach relies on the mixture or Choquet representation of consistent scoring functions discussed by \citet{EhmGneitingETAL2016}; see Section \ref{sec:def} for the necessary definitions. In Section \ref{sec:Pareto}, we introduce Pareto optimal parameters and illustrate their properties. The paper closes with a discussion of the results in Section \ref{sec:discussion}.

\section{Identifiable functionals and mixture representations} \label{sec:def}

Following \citet[Definition 1]{JordanMuhlemannETAL2022}, we define an identification function as a function $V:\R \times \R \to \R$ such that $V(\cdot,y)$ is increasing and left-continuous for all $y \in \R$. For any probability measure $P$ on $\R$ such that $\bar{V}(z,P) = \int V(z,y)\diff P(y)$ exists for all $z \in \R$, we define the functional $T$ induced by $V$ as
\[
T(P) = [T_P^-,T_P^+] \subseteq [-\infty,\infty],
\]
where
\[
T_P^- = \sup\{z \mid \bar{V}(z,P) < 0\} \quad \text{and}\quad T_P^+ = \inf\{z \mid \bar{V}(z,P) > 0\}. 
\]
If $T_P^+ = T_P^-$ for all relevant probability measures $P$, we identify $T(P)$ with this value. 

\begin{example}\label{ex:Vfunc}
For the mean, we have the identification function $V(z,y) = z - y$, for the second moment, $V(z,y) = z - y^2$ is an identification function, and for the $\alpha$-quantiles, one can take the identification function $V(z,y) = \mathbbm{1}\{y < z\} - \alpha$. 
\end{example}
Further examples including some robust location functionals are given in \citet[Table 1]{JordanMuhlemannETAL2022}.


For a given identification function $V$, we define the elementary loss functions as 
\begin{equation}\label{eq:generalElem}
S_\eta(z,y) = (\mathbbm{1}\{\eta \le z\} - \mathbbm{1}\{\eta \le y\})V(\eta,y)
\end{equation}
for $\eta,z,y\in\R$. By \citet[Proposition 1]{JordanMuhlemannETAL2022}, the elementary loss functions $S_\eta$ are consistent relative to the class $\mathcal{P}$ of all probability measure $P$ such that $\bar{V}(\eta,P)$ exists, that is, for all $P \in \mathcal{P}$, $t \in T(P)$, and $z \in \R$ it holds that
\begin{equation}\label{eq:consistent}
\E S_\eta(t,Y) \le \E S_\eta(z,Y),
\end{equation}
where $Y \sim P$, that is, $Y$ has distribution $P$. As a consequence, also all loss functions for the form 
\begin{equation}\label{eq:Choquet}
L_H(z,y) = \int_{-\infty}^{\infty} S_{\eta}(z,y) \diff H(\eta), \quad z,y \in \R,
\end{equation}
for some positive measure $H$ on $\mathbb{R}$, are consistent for $T$, that is, \eqref{eq:consistent} holds with $S_\eta$ replaced by the loss function $L_H$. 

By Osband's principle, a mixture representation such as \eqref{eq:Choquet} is always available for sufficiently regular loss functions if the functional $T$ is identifiable with identification function $V$; see \citet{Gneiting2011,SteinwartPasinETAL2014,Ziegel2016}. In particular, all loss functions for the functionals in Example \ref{ex:Vfunc} and in \citet[Table 1]{JordanMuhlemannETAL2022} that are consistent with respect to the maximal possible domain of definition of the respective functionals are of the form \eqref{eq:Choquet}. For quantiles and expectiles, this Choquet or mixture representation of consistent loss functions was discussed in detail by \citet{EhmGneitingETAL2016}.

\section{T-calibrated models}\label{sec:T-cal}

Let $(X,Y)$ be a pair of random variables where $X$ takes values in $\mathcal{X}$ and $Y$ is real-valued. Let $V$ be an identification function and assume that all conditional distributions $\mathcal{L}(Y \mid X = x)$ are contained in the class $\mathcal{P}$ of distributions $P$ such that $\bar{V}(z,P)$ exists for all $z \in \R$. Let $T$ be the functional induced by $V$.  

Consider a model $m(x;\beta)$, $\beta \in \Theta$, for 
\[
T(Y|X=x) = T(\mathcal{L}(Y|X=x)), \quad x \in \mathcal{X},
\]
where we slightly abuse notation and abbreviate $T(\mathcal{L}(Y|X=x))$ to $T(Y|X=x)$. The model is correctly specified if $m(X;\beta^\ast) \in T(Y|X)$ a.s.~for some $\beta^\ast \in \Theta$. Inspired by \citet{GneitingResin2023}, we define T-calibration as follows.
\begin{definition}
    The model is \emph{T-calibrated} if
    \[
    m(X;\beta^*) \in T(Y | m(X;\beta^*)) \quad \text{a.s.~for some $\beta^\ast \in \Theta$.}
    \]
\end{definition}

Let $L_H$ be a consistent loss function for $T$ of the form \eqref{eq:Choquet}. We always have
\begin{equation}\label{eq:1}
\E L_H(g(X),Y) \le \inf_{\beta\in\Theta} \E L_H(m(X;\beta),Y),
\end{equation}
if $g(X) \in T(Y|X)$ almost surely. 
If the model is correctly specified, that is, $m(X,\beta^\ast) \in \in T(Y|X)$ almost surely for some $\beta^* \in \Theta$, then we have equality in \eqref{eq:1} and the infimum is in fact a minimum that is attained at $\beta^*$. This holds independently of which loss function $L_H$ we have chosen. 
Let $\beta(H)$ be such that $\min_{\beta\in\Theta} \E L_H(m(X;\beta),Y) = \E L_H(m(X;\beta(H)),Y)$. 

Conversely, we would like to show that if the infimum in \eqref{eq:1} is attained at $\beta(H)$ and $\beta(H) = \beta^*$ for any measure $H$ and associated loss function $L_H$, then we have $ m(X;\beta^*) \in T(Y|X)$ almost surely. This conjecture is generally false. Suppose for example that $\Theta$ only has one element $\Theta=\{\beta^*\}$. Then we always trivially have $\beta(\phi) = \beta^*$ but there is no good reason why $m(\cdot\,;\beta^*)$ should be correctly specified. The same problem appears if for example $X = (X_1,\dots,X_k)$ and $T(Y|X=x)$ depends on $x_k$ but $m(x;\beta^*) = T(Y|(X_1,\dots,X_{k-1})=(x_1,\dots,x_{k-1}))$, or, in other words: As soon as all $m(X;\beta)$ are measurable with respect to a strict sub-$\sigma$-algebra $\mathcal{A}$ of $\sigma(X)$ and the optimal predictor with respect to $\mathcal{A}$ is in our model (and has parameter $\beta^*$), then we will have $\beta(H) = \beta^*$ for all $H$. If $T(Y|X)$ is not $\mathcal{A}$ measurable, the conjecture cannot hold. Both counter-examples are not surprising. 

Instead, we can show instead that if there is a unique optimal parameter under all consistent loss functions $L_H$, then the model is T-calibrated. In fact, we only need to assume that there a parameter $\beta^* \in \Theta$ such that the infimum in \eqref{eq:1} is attained at $\beta^*$ for any elementary loss function $L_H = S_\eta$, $\eta \in \R$ as defined at \eqref{eq:generalElem}. This simplifies things in particular with regards to integrability assumptions in the formulation of Theorem \ref{thm:1}.

We make the following assumption on the model, which is always satisfied if the model allows for an arbitrary intercept. 
\begin{assumption}\label{assu:2}
    Let $\epsilon > 0$.  Suppose that for any model $m(\cdot\,;\beta)$ and $a \in (-\epsilon,\epsilon)$, there is $\beta' \in \Theta$ such that $m(\cdot\,;\beta) + a = m(\cdot\,;\beta')$.  
\end{assumption}

\begin{theorem}\label{thm:1}
Suppose that Assumptions \ref{assu:2} holds, and that there exists $\beta^* \in \Theta$ such that $\inf_{\beta \in \Theta} \E S_\eta(m(X;\beta),Y)= \E S_\eta(m(X;\beta^*),Y)$ for all $\eta \in \R$. 
Assume that the random variables $V(m(X;\beta^*)\pm \epsilon',Y)$ are integrable for some $\epsilon' > 0$. 
Then \[m(X;\beta^*) \in T(Y|m(X;\beta^*)).\] 
\end{theorem}

\begin{remark}
If $\sigma(m(X;\beta^*)) \supseteq \sigma(g(X))$, then Theorem \ref{thm:1} implies that $m(X;\beta^*) = g(X)$. For example, this happens in a linear model $a + bX$, $X \in \mathbb{R}$ with $\beta^* \not=(a,0)$ for any $a$. It is also true if $X$ is discrete with values $\{x_1,\dots,x_K\}$ and all values of $m(x_i;\beta^*)$ are distinct. 
\end{remark}

For the proof of Theorem \ref{thm:1}, we need the following lemma. 

\begin{lemma}\label{lem:2}
Suppose that Assumptions \ref{assu:2} holds, and that there exists $\beta^* \in \Theta$ such that $\inf_{\beta \in \Theta} \E S_\eta(m(X;\beta),Y)= \E S_\eta(m(X;\beta^*),Y)$ for all $\eta \in \R$. Then, for all $a \in (0,\epsilon)$, $\eta \in \R$, we have
\begin{align*}
\E&\mathbbm{1}\{m(X;\beta^*) \in [\eta,\eta + a\}V(\eta + a,Y) \ge 0,\\
\E&\mathbbm{1}\{m(X;\beta^*) \in [\eta,\eta + a\}V(\eta,Y) \le 0.
\end{align*}
\end{lemma}

\begin{proof}
We have for any $a \in (0,\epsilon)$ and $\eta \in U$
\[
\E(\mathbbm{1}\{\eta \le m(X;\beta^*)\}-\mathbbm{1}\{\eta \le m(X;\beta^*) - a\})V(\eta,Y) \le 0.
\]
This is equivalent to 
\begin{equation*}
\E\mathbbm{1}\{m(X;\beta^*) \in [\eta,\eta+a)\}V(\eta,Y) \le 0.
\end{equation*}
On the other hand, we also obtain
\[
\E(\mathbbm{1}\{\eta + a \le m(X;\beta^*)\}-\mathbbm{1}\{\eta + a \le m(X;\beta^*) + a\})V(\eta + a,Y) \le 0,
\]
which is equivalent to 
\begin{equation*}
\E\mathbbm{1}\{m(X;\beta^*) \in [\eta,\eta+a)\}V(\eta + a,Y) \ge 0.
\end{equation*}
\end{proof}

\begin{proof}[Proof of Theorem \ref{thm:1}]
Let $\theta_{n,k} = k2^{-n}$ for $k \in \mathbb{Z}$, $n \in \mathbb{N}$. Define
\begin{align*}
Z_n &:= \sum_{k \in \mathbb{Z}} \theta_{n,k} \mathbbm{1}\{m(X;\beta^*) \in [\theta_{n,k},\theta_{n,k+1})\},\\
W_n &:= \sum_{k \in \mathbb{Z}} \theta_{n,k+1} \mathbbm{1}\{m(X;\beta^*) \in [\theta_{n,k},\theta_{n,k+1})\},
\end{align*}
and 
\[
\mathcal{A}_n := \sigma(Z_n) = \sigma(\{m(X;\beta^*) \in [\theta_{n,k},\theta_{n,k+1})\}, k \in \mathbb{Z}).
\]
We have $Z_n \le Z_{n+1}$, $W_n \ge W_{n+1}$ and thus by the monotonicity of $V(\cdot,y)$, we obtain $V(Z_n,Y) \le V(Z_{n+1},Y)$ and $V(W_n,Y) \ge V(W_{n+1},Y)$. Furthermore, $Z_n \uparrow m(X;\beta^*)$ and $W_n \downarrow m(X;\beta^*)$ as $n \to \infty$. Note that $\mathcal{A}_n \subset \mathcal{A}_{n+1}$ because 
\begin{multline*}
\{m(X;\beta^*) \in [\theta_{n,k},\theta_{n,k+1})\} \\= \{m(X;\beta^*) \in [\theta_{n+1,2k},\theta_{n+1,2k+1})\} \cup \{m(X;\beta^*) \in [\theta_{n+1,2k+1},\theta_{n+1,2k+2})\}.
\end{multline*}

We define $\bar{Z}_n := \E(V(Z_n,Y)|\mathcal{A}_n)$ and $\bar{W}_n := \E(V(W_n,Y)|\mathcal{A}_n)$. For $n$ large enough, Lemma \ref{lem:2} implies that $\bar{Z}_n \le 0$ and $\bar{W}_n \ge 0$ as the generator of $\mathcal{A}_n$ consists of disjoint sets. Furthermore, 
\begin{align*}
\E(\bar{Z}_{n+1}|\mathcal{A}_n) &= \E(\E(V(Z_{n+1},Y)|\mathcal{A}_{n+1})|\mathcal{A}_n) \\
&= \E(V(Z_{n+1},Y)|\mathcal{A}_n)\\
&\ge \E(V(Z_n,Y)|\mathcal{A}_n) = \bar{Z}_n,
\end{align*}
and, analogously, $\E(\bar{W}_{n+1}|\mathcal{A}_n) \le \bar{W}_n$. Therefore $(\bar{Z}_n)_n$ is a non-positive sub-martingale and $(\bar{W}_n)_n$ is a non-negative super-martingale with respect to $(\mathcal{A}_n)_n$. \citet[60.2 Korrolar 1]{Bauer1974} implies that there exists $\bar{Z}_{\infty} \le 0$ and $\bar{W}_{\infty} \ge 0$ integrable such that $\bar{Z}_n \to \bar{Z}_{\infty}$, $\bar{W}_n \to \bar{W}_{\infty}$ almost surely as $n \to \infty$ and $\E(\bar{Z}_{\infty}|\mathcal{A}_n) \ge \bar{Z}_n$, $\E(\bar{W}_{\infty}|\mathcal{A}_n) \le \bar{W}_n$.

The continuity and monotonicity of $V(\cdot,Y)$ yields that $V(Z_n,Y) \uparrow V(m(X;\beta^*),Y)$ and $V(W_n,Y) \downarrow V(m(X;\beta^*)+,Y)= \lim_{h \downarrow 0} V(m(X;\beta^*)+h,Y)$ almost surely as $n \to \infty$. 

L\'evy's Zero-One-Law yields that 
\[
0 \le \E(V(m(X;\beta^*),Y) - V(Z_n,Y)\mid \mathcal{A}_n) \to \E(V(m(X;\beta^*),Y)|\mathcal{A}_{\infty}) - \bar{Z}_\infty
\]
almost surely as $n \to \infty$, where $\mathcal{A}_{\infty} = \bigcup_n \mathcal{A}_n = \sigma(m(X;\beta^*))$. Using the integrability assumption, we obtain that
\[
0 \le \E(V(m(X;\beta^*),Y) - V(Z_n,Y)) \to 0,
\]
hence
\[
\E(V(m(X;\beta^*),Y)|m(X;\beta^*)) = \bar{Z}_\infty \le 0. 
\]
Analogous arguments show that
\[
\E(V(m(X;\beta^*)+,Y)|m(X;\beta^*)) = \bar{W}_\infty \ge 0,
\]
which yields the claim.
\end{proof}

\section{Pareto-optimal parameters}\label{sec:Pareto}

We consider the same setting as in Section \ref{sec:T-cal}, that is $T$ is an identifiable functional induced by the identification function $V$. We consider consistent loss functions $L_H$ of the form \eqref{eq:Choquet}. 

Typically, it will happen that there is no parameter $\beta^*$ for the model $m(x;\beta)$, $\beta \in \Theta$ that is preferred under all loss functions $L_H$. If we consider all loss functions to be important, it is natural to study the parameters of the model that are Pareto optimal with respect to all loss functions. In this section, we make the notion of Pareto optimal parameters precise, study some of their elementary properties, and give examples.

\begin{definition}
A parameter $\beta_1 \in \Theta$ is \emph{dominated} by a parameter $\beta_2 \in \Theta$ if 
\[
\E L_H(m(X;\beta_2),Y) \le \E L_H(m(X;\beta_1),Y)
\]
for all measures $H$. It is \emph{strictly dominated} if the inequality is strict for some measure $H$.
A parameter $\beta\in \Theta$ is \emph{Pareto optimal} if it is not strictly dominated by any other parameter.
\end{definition}
In other words, a parameter $\beta$ is Pareto optimal if for all $\beta'\not=\beta$, either, there exists a $H$ such that
\begin{equation*}
\E L_H(m(X;\beta),Y) < \E L_H(m(X;\beta'),Y),
\end{equation*}
or, for all $H$, we have
\begin{equation*}
\E L_H(m(X;\beta),Y) = \E L_H(m(X;\beta'),Y).
\end{equation*}
It is clear that a parameter that is the unique minimizer of some consistent loss function $L_H$ is Pareto optimal. Also, any parameter that minimizes all loss functions simultaneously is Pareto optimal. Furthermore, in this case, if the model is correctly specified with a unique true parameter $\beta^* \in \Theta$, this is the only Pareto optimal parameter.

Dominance and Pareto optimality can also be formulated in terms of the elementary loss functions $S_\eta$ as defined at \eqref{eq:generalElem}. A parameter $\beta_1\in\Theta$ is dominated by $\beta_2\in\Theta$ if for all $\eta \in \mathbb{R}$, we have
\[
\E S_\eta(m(X;\beta_2),Y) \le \E S_\eta(m(X;\beta_1),Y)
\]
for all $\eta$. If the inequality is strict for some $\eta$, then $\beta_2$ strictly dominates $\beta_1$. 

A parameter $\beta \in \Theta$ is Pareto optimal if for all $\beta'\not=\beta$, either, there exists an $\eta$ such that
\begin{equation*}
\E S_\eta(m(X;\beta),Y) < \E S_\eta(m(X;\beta'),Y),
\end{equation*}
or, for all $\eta$, we have
\begin{equation*}
\E S_\eta(m(X;\beta),Y) = \E S_\eta(m(X;\beta'),Y).
\end{equation*}

Generally, there are more Pareto optimal parameters than unique minimizers of consistent loss functions. This is a classical fact in multi-objective optimization \citep{Geoffrion1968}. Since the map 
\[
\Theta \times \mathbb{R} \to \mathbb{R}, \quad (\beta,\eta) \mapsto f(\beta,\eta) = \E S_\eta(m(X;\beta),Y)
\]
is typically not convex, it is a hard problem to characterize which Pareto optimal parameters are minimizers of consistent loss functions.

We conclude this section with some examples where the Pareto optimal parameters of a model can be computed explicitly.

\begin{example}
Suppose that $X$ and $\epsilon$ are independent and standard normally distributed. Let
\[
Y = X^2 + \epsilon,
\]
and consider $m(x;b) = bx$, $b \in \mathbb{R}$, as a model for the conditional mean of $Y$. For $\eta \in \R$, define $b(\eta) = \arg\min_b \E S_\eta(Y,bX) = \arg\min_b \E(\mathbbm{1}\{\eta \le bX\}(\eta - Y))$. We obtain
\begin{align*}
\E(\mathbbm{1}\{\eta \le bX\}(\eta - Y)) &= \E(\mathbbm{1}\{\eta \le bX\}(\eta - X^2))
= \begin{cases}\int_{\eta/b}^{\infty}(\eta - x^2) \varphi(x) dx, & b > 0,\\
\int_{-\infty}^{\eta/b}(\eta - x^2) \varphi(x) dx, & b < 0,\\
\mathbbm{1}\{\eta \le 0\}(\eta - \E(X^2)), & b = 0.\end{cases}
\end{align*}
For $b > 0$, we obtain
\[
\frac{d}{db}\int_{\eta/b}^{\infty}(\eta - x^2) \varphi(x) dx = \frac{(b^2 - \eta)\eta^2}{b^4}\varphi(\eta/b),
\]
which is increasing in $b$ for $\eta < 0$ and has a local minimum at $b = \sqrt{\eta}$ for $\eta > 0$. For $b < 0$, we have
\[
\frac{d}{db}\int_{-\infty}^{\eta/b}(\eta - x^2) \varphi(x) dx = \frac{(\eta-b^2)\eta^2}{b^4}\varphi(\eta/b),
\]
which is decreasing in $b$ for $\eta < 0$ and has a local minimum at $b = -\sqrt{\eta}$ for $\eta > 0$. At the local minimum the function takes the value
\[
\int_{-\infty}^{-\sqrt{\eta}}(\eta - x^2) \varphi(x) dx  = \int_{\sqrt{\eta}}^{\infty}(\eta - x^2) \varphi(x) dx. 
\]
Therefore, 
\[
b(\eta) = \begin{cases} \{-\sqrt{\eta},\sqrt{\eta}\},&\eta > 0,\\
0,&\eta \le 0,\end{cases}
\]
and the set of Pareto optimal parameters is the whole real line $\R$. 
\end{example}

\begin{example}\label{ex:Anja}
This  example is taken from \citet[Example 4.4.3]{Muhlemann2021}. 
    We consider $X,Y \in \R$ and the model $m(x;(\beta_0,\beta_1)) = \beta_0 + \beta_1 x$, $(\beta_0,\beta_1) \in \R \times (0,\infty)$ for the conditional mean of $Y$ given $X$. Suppose first that
    \[
    \E(Y |X=x) = g_1(x) = \frac{e^x}{1+e^x}, \quad x \in \R. 
    \]
    Then, the set of Pareto optimal parameters is given by all $\beta=(\beta_0,\beta_1)$ of the form
    \[
    \beta_0 = \eta\left(1 - (1-\eta)\log\left(\frac{\eta}{1-\eta}\right)\right), \quad \beta_1 = \eta(1-\eta), \quad \eta \in [0,1],
    \]
    and all $\beta$ such that $\eta \mapsto \eta - g_1((\eta-\beta_0)/\beta_1)$ has at least two zeros. As a second example we consider
   \[
    \E(Y |X=x) = g_2(x) = x^3, \quad x \in \R. 
    \]
    Then, the set of Pareto optimal parameters is given by all $\beta=(\beta_0,\beta_1)$ of the form
    \[
    \beta_0 = -2\eta, \quad \beta_1 = 3\eta^{2/3}, \quad \eta \in \R,
    \]
    and all $\beta$ such that $\eta \mapsto \eta - g_1((\eta-\beta_0)/\beta_1)$ has at least two zeros. Both sets of Pareto optimal parameters are illustrated in Figure \ref{fig:1}.
\end{example}
\begin{figure}
    \centering
    \includegraphics[width=0.45\linewidth]{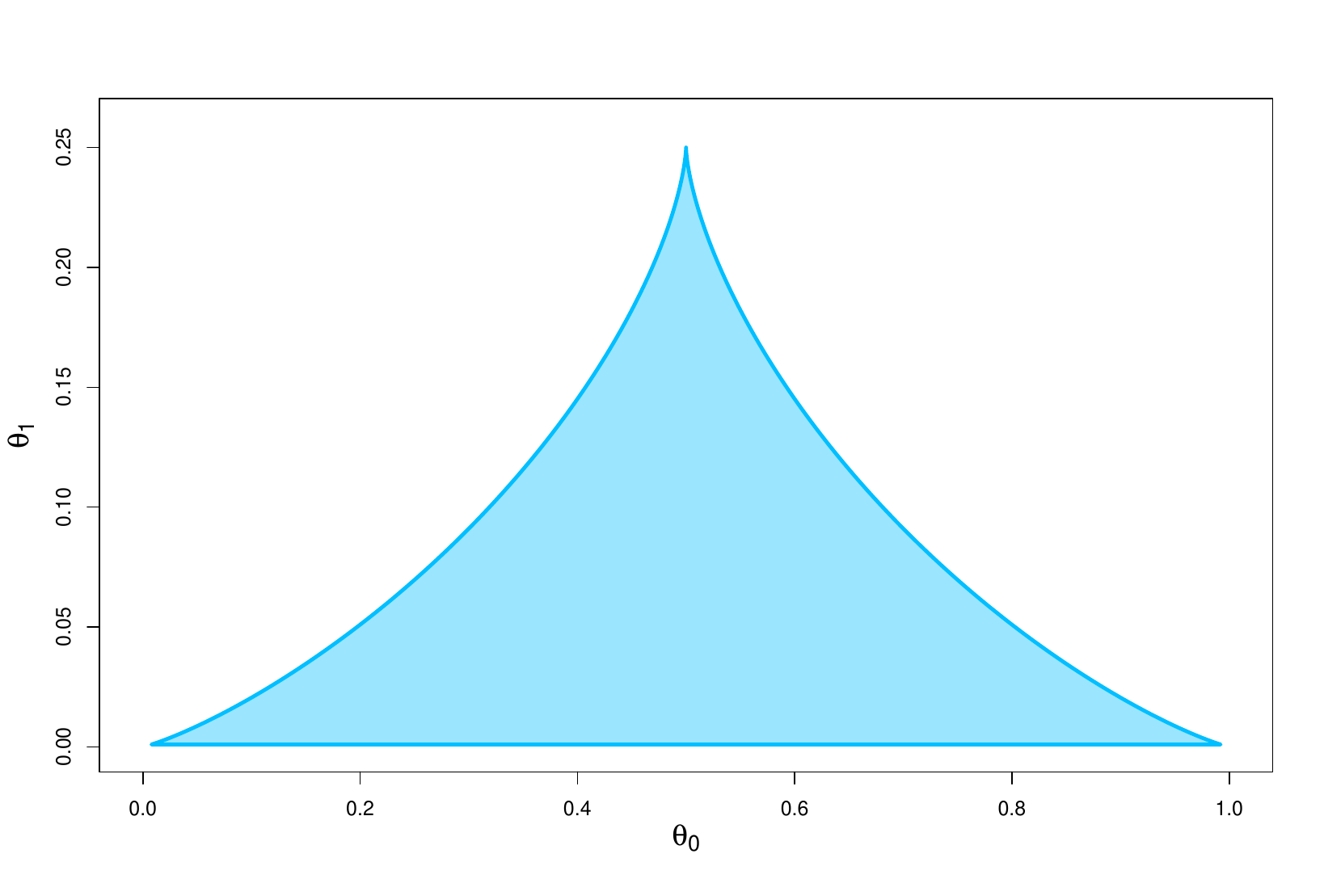}
    \includegraphics[width=0.45\linewidth]{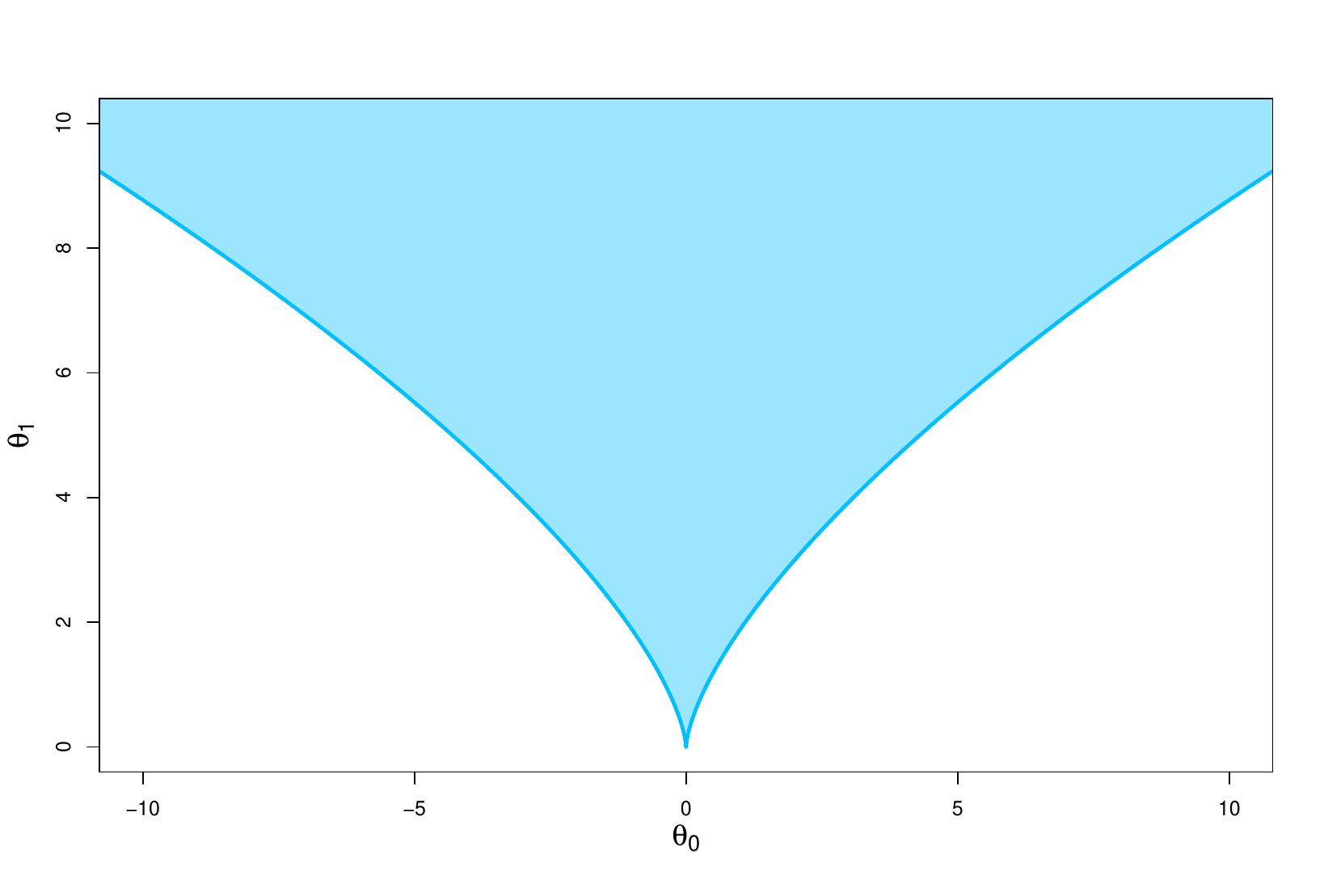}
    \caption{Pareto optimal parameters in blue for the model in Example \ref{ex:Anja} and regression function $g_1$ and $g_2$ in the left and right panel, respectively.}
    \label{fig:1}
\end{figure}

\section{Discussion}\label{sec:discussion}

We have considered parameter estimation in semi-parametric models for identifiable functionals under the entire class of consistent loss functions. This approach differs from classical approaches in statistics: We fix the functional that we are modeling and allow for all consistent loss functions. Classically, there is always a fixed loss.

It turns out that a model is T-calibrated if one parameter of the model is preferred under all consistent loss functions. If there is no single preferred parameter under all consistent loss functions, we studied properties of the set of all Pareto optimal parameters. 

Intuitively, if the model is not correctly specified, the set of Pareto optimal parameters should inform how severe the misspecification is, and possibly also give some indication how the model specification could be improved. Unfortunately, we found this challenging even for conditional mean models with a one-dimensional covariate $X$ and an increasing regression function. Furthermore, in order for any procedure based on Pareto optimal parameters to be practically useful, the set of Pareto optimal parameters has to be estimated from finite samples, which is computationally challenging. 

Nevertheless, we believe that Theorem \ref{thm:1} is of theoretical interest to understand the role of the loss function in regression problems. 


\bibliographystyle{plainnat}
\bibliography{biblio}

\end{document}